\documentclass[12pt,reqno]{amsart}
\usepackage{amscd,amssymb,amsthm}
\usepackage{graphicx}

\usepackage{enumerate}
\theoremstyle{plain}
\newtheorem{theorem}{Theorem}[section]
\newtheorem{corollary}[theorem]{Corollary}

\newtheorem{example}[theorem]{Example}
\newtheorem{fexample}[theorem]{Further Examples}
\newtheorem{proposition}[theorem]{Proposition}
\theoremstyle{remark}

\numberwithin{equation}{section}

\newcommand{\reel}{\mathbb{R}}

\newcommand{\ent}{\mathbb{Z}}
\newcommand{\comp}{\mathbb{C}}
\newcommand{\tore}{\mathbb{T}}
\newcommand{\ds}{\displaystyle}

\newcommand{\abs}[1]{\left\vert #1\right\vert }
\newcommand{\adh}[1]{\overline{#1} }
\newcommand{\N}[1]{{\left\Vert#1\right\Vert}}

\newcommand{\bg}{\medskip\goodbreak}
\newcommand{\impl}{{\quad\Longrightarrow\quad}}
\newcommand{\vers}{{\,\longrightarrow\,}}
\DeclareMathOperator{\sgn}{sgn}

\DeclareMathOperator{\li}{Li}
\newcommand{\itemref}[1]{\eqref{#1}}

\newenvironment{enumeratei}{\begin{enumerate}%
	[\itshape i.]}{\end{enumerate}}

\title[Exact Evaluation Of Some Highly Oscillatory Integrals]
{Exact Evaluation Of Some Highly Oscillatory Integrals}
\author[Omran Kouba]{Omran Kouba$^\dag$}
\address{Department of Mathematics \\
Higher Institute for Applied Sciences and Technology\\
P.O. Box 31983, Damascus, Syria.}
\email{omran\_kouba@hiast.edu.sy}
\thanks{$^\dag$ Department of Mathematics, Higher Institute for Applied Sciences and Technology.}

\begin{document}
\parindent=0pt
%\date{\today}
\begin{abstract}
In this note a general result is proved that can be used  to evaluate exactly a class of highly
oscillatory integrals. \par
\end{abstract}
\smallskip\goodbreak

\parindent=0pt
\maketitle
%%%%%%%%%%%%%%%%%%%%%%%%
\section{\sc Introduction }\label{sec1}
\bg
\parindent=0pt
\qquad The first of the ten $\$100$,$100$-digit challenges \cite{ter,siam},  proposed to calculate 
the integral $\int_0^1x^{-1}\cos(x^{-1}\log x)\,dx$ to ten digits. This was a real challenge since
the integrand oscillates infinitely often inside the interval of integration.
\bg
\qquad In this note, a general result is proved that will allow us to determine exactly the value
 of some highly oscillating integrals.
To give you the flavour of what we will prove, here is one of our results :
\begin{equation}\label{E:eq1}
\int_0^{\pi/2}\frac{d\theta}{1+\sin^2(\tan\theta)}
=\frac{\pi}{2\sqrt{2}}\left(\frac{e ^2+3-2\sqrt2}{ e^2-3+2\sqrt2}\right),
\end{equation}
where the graph of the integrand is depicted in the following Figure \ref{fig1}.

\begin{figure}[!h]
\begin{center}
\includegraphics[width=0.5\textwidth]{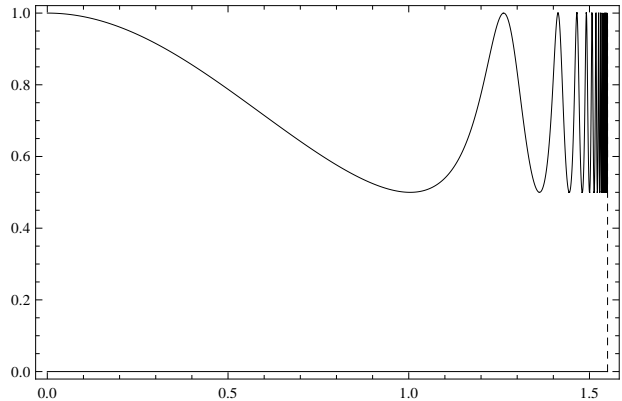}
\caption{The graph of the function $x\mapsto\frac{1}{1+\sin^2(\tan x)}$ on $\left[0,\frac{\pi}{2}\right)$.}\label{fig1}
\end{center}
\end{figure}

\qquad In section \ref{sec2} we will prove our main results 
and in section \ref{sec4} we will give some detailed examples and applications.
\bg

%%%%%%%%%%%%%%%%%%%%%%%%%%%%%%%%%%%
\bg

\section{\sc The Main Results}\label{sec2}
\nobreak

\qquad First, let us set the framework of our investigation. Our starting point will be 
a $2\pi$-periodic locally integrable function $f\in L^1(\tore)$.
For a nonzero real number
$\lambda$, we are interested in the evaluation of the integrals:
\begin{align}
\mathcal{I}_\lambda(f) &=\int_0^\infty\frac{f_e(\theta)}{\lambda^2+\theta^2}\,d\theta
=\frac{1}{\abs{\lambda}}\int_0^{\pi/2}f_e(\lambda \tan x)\,dx,  \label{E:4I}\\
\noalign{\text{and}}\notag\\
\mathcal{J}_\lambda(f) &=\int_0^\infty\frac{\theta f_o(\theta)}{\lambda^2+\theta^2}\,d\theta
=\int_0^{\pi/2}f_o(\abs{\lambda} \tan x)\tan x\,dx, \label{E:4J}
\end{align}
where $f_e$ and $f_o$ are, respectively, the even and the odd components 
of $f$. They
are defined on $\reel$ by the formul\ae:
\begin{equation}\label{E:4feo}
f_e(t) =\frac{f(t)+f(-t)}{2} \quad\text{and} \quad
f_o(t) =\frac{f(t)-f(-t)}{2}.
\end{equation}
\bg
\qquad On the other hand,  the Cauchy principal value $\mathcal{L}_\lambda(f)$ of the integral
$\int_{-\infty}^{\infty}\frac{f(t)}{t+i\lambda}\,d t$, plays an important
role  in this study. Recall that $\mathcal{L}_\lambda(f)$ is defined as follows:

\begin{equation}\label{E:3L}
\mathcal{L}_\lambda(f)=\lim_{a\to\infty}\int_{-a}^{a}\frac{f(t)}{t+i\lambda}\, d t
\end{equation}

\qquad In the next proposition, we prove the convergence of the integrals defined by \eqref{E:4I} and \eqref{E:4J},
and we describe their relationship to the principal value defined by \eqref{E:3L}.\bg

%%%%%%%%%%%%%%%%%%%%%%%%%%%%%%%%%%%%%%%%%%
\begin{proposition} \label{pr41}
Consider $f\in L^1(\tore)$, and a  nonzero real $\lambda$. 
\begin{enumeratei}
\item \label{pr411} The integral $\ds \int_0^\infty\frac{f(t)}{t^2+\lambda^2}\,dt$ is absolutely convergent.
\item \label{pr412} If $\int_\tore f(t)\,dt=0$, then the integral
$\ds \int_0^\infty\frac{t f(t)}{t^2+\lambda^2}\,dt$ is convergent.
\item \label{pr413} The principal value $\mathcal{L}_\lambda(f)$ is well defined and
\begin{equation}\label{E:pr41}
\mathcal{L}_\lambda(f)= 2\left(\mathcal{J}_\lambda(f)-i\lambda \mathcal{I}_\lambda(f)\right). 
\end{equation}
where $\mathcal{I}_\lambda(f)$ and $\mathcal{J}_\lambda(f)$ are defined by \eqref{E:4I} and \eqref{E:4J}, respectively.
\end{enumeratei}
\end{proposition}
\begin{proof}
Indeed, the first point is easy since, for $k>0$, we have
\[
\int_{2 k\pi}^{2(k+1)\pi}\frac{\abs{ f(t)}}{t^2+\lambda^2}\,dt=
\int_{0}^{2\pi}\frac{ \abs{f(t)}}{(t+2k\pi)^2+\lambda^2}\,dt\leq\frac{1}{4k^2\pi^2}\int_{0}^{2\pi} \abs{f(t)}\,dt
\] 
So
\[
\int_0^\infty\frac{\abs{ f(t)}}{t^2+\lambda^2}\,dt
\leq\left(\frac{1}{\lambda^2}+\frac{1}{4\pi^2}\sum_{k=1}^\infty\frac{1}{k^2}\right)\int_{0}^{2\pi} \abs{f(t)}\,dt<+\infty,
\]
and \itemref{pr411} follows.\bg
\qquad To prove \itemref{pr412} we consider the function $F$ defined by $F(x)=\int_0^xf(t)\,dt$. Since
$f$ belongs to $L^1(\tore)$ and $\int_\tore f(t) dt=0$ we conclude immediately that
$F$ is continuous and $2\pi$-periodic. In particular, $F$ is bounded  on $\reel$. \bg
\qquad An integration by parts shows that, for $x>0$, we have
\[\int_0^x\frac{t}{\lambda^2+t^2}f(t)\,dt=\frac{x}{\lambda^2+x^2}F(x)-
\int_0^x\frac{\lambda^2-t^2}{(\lambda^2+t^2)^2}F(t)\,d t
\]
Indeed, this version of ``integration by parts'' is a dirct application of Fubini's theorem, see for instance \cite[Theorem~5.2.3]{ath}. 
Now, the following inequality
\[
\forall\,t\geq0,\quad\abs{\frac{\lambda^2-t^2}{(\lambda^2+t^2)^2}F(t)}\leq \frac{M}{\lambda^2+t^2} 
\qquad \hbox{where $M=\sup_{\reel}\abs{F}$,}
\]
proves the absolute convergence of the integral $\ds \int_0^\infty\frac{\lambda^2-t^2}{(\lambda^2+t^2)^2}F(t)dt$.
This implies the convergence of $\ds \int_0^\infty\frac{t}{\lambda^2+t^2}f(t)dt$ and shows that
\[
\int_0^\infty\frac{t}{\lambda^2+t^2}f(t)dt=\int_0^\infty\frac{t^2-\lambda^2}{(\lambda^2+t^2)^2}F(t)\,dt.
\]

\qquad On the other hand, for $a>0$, we have
\begin{align*}
\int_{-a}^a\frac{f(t)}{t+i\lambda}\,d t
&=\int_{-a}^a\frac{(t-i\lambda)(f_e(t)+f_o(t))}{\lambda^2+t^2}\,dt\\
&=\int_{-a}^a\frac{t f_e(t)-i\lambda f_o(t)}{\lambda^2+t^2}\,d\ t
+\int_{-a}^a\frac{t f_o(t)-i\lambda\,f_e(t)}{\lambda^2+t^2}\,dt\\
&=2\int_{0}^a\frac{t f_o(t)-i\lambda\,f_e(t)}{\lambda^2+t^2}\,dt
\end{align*}
because $t\mapsto  t f_e(t)-i\lambda f_o(t)$ is odd,
 and $t\mapsto  t f_o(t)-i\lambda\,f_e(t)$ is even. \bg
\qquad Finally, noting that $\int_\tore f_o(t)dt=0$ since $f_o$ is odd, and using \itemref{pr411} and \itemref{pr412}, we
obtain \eqref{E:pr41} by letting $a$ tend to $+\infty$, and the Proposition follows.
\end{proof}

\bg
\qquad Using the notation of the preceding proposition,  and the parity of $\lambda\mapsto\mathcal{I}_\lambda(f)$ and
$\lambda\mapsto\mathcal{J}_\lambda(f)$, we obtain the following corollary that reduces the determination of 
$\mathcal{I}_\lambda(f)$ and $\mathcal{J}_\lambda(f)$ to that of $\mathcal{L}_\lambda(f)$.

\bg
\begin{corollary}\label{co41} Using the notation of Proposition \ref{pr41}, we have
\begin{align}
\mathcal{I}_\lambda(f)&=\frac{i}{4\lambda}\left(\mathcal{L}_{\lambda}(f)-\mathcal{L}_{-\lambda}(f)\right),\label{E:co411}\\
\mathcal{J}_\lambda(f)&=\frac{1}{4}\left(\mathcal{L}_{\lambda}(f)+\mathcal{L}_{-\lambda}(f)\right).\label{E:co412}
\end{align}
\end{corollary}
\bg

\qquad Thus, the question is reduced to determining $\mathcal{L}_{\lambda}(f)$. The next proposition expresses differently this quantity.
\bg      

%%%%%%%%%%%%%%%%%%%%%%%%%%%%%%%%%%%%%%%%%%%
\begin{proposition}\label{pr42} For $f\in L^1(\tore)$, and a  nonzero real $\lambda$, we have
\begin{equation} 
\mathcal{L}_\lambda(f)=\frac{1}{2}\int_0^{2\pi}\cot\left(\frac{t+i\lambda}{2}\right)\,f(t)dt
\end{equation}
where $\mathcal{L}_\lambda(f)$ is defined by \eqref{E:3L}.
\end{proposition}
\begin{proof}
Recall that,(see \cite[pages~187-190]{ahl},)
$$\cot(z)=\frac{1}{z}+\sum_{n=1}^\infty\left(\frac{1}{z-\pi n}+\frac{1}{z+\pi n}\right)=\lim_{n\to\infty}\sum_{k=-n}^{n-1}\frac{1}{z+\pi k},$$
with normal convergence on every compact set $K$ contained in $\comp\setminus \pi\ent$. 
Applying this to the compact segment $K=\left\{\frac{t+i\lambda}{2}:t\in[0,2\pi]\right\}$, and recalling that $f\in L^1(\tore)$, we conclude that
\begin{align*}
\int_0^{2\pi}\cot\left(\frac{t+i\lambda}{2}\right)\,f(t)dt
&=\lim_{n\to\infty}\sum_{k=-n}^{n-1}\left(\int_{0}^{2\pi}\frac{2f(t)}{t+2\pi k+i\lambda}\, dt\right),\\
&=2\lim_{n\to\infty}\sum_{k=-n}^{n-1}\left(\int_{2\pi k}^{2\pi(k+1)}\frac{f(t)}{t+i\lambda}\, dt\right),\\
&=2\lim_{n\to\infty}\int_{-2\pi n}^{2\pi n}\frac{f(t)}{t+i \lambda}\, dt,\\
&=2\mathcal{L}_\lambda(f).
\end{align*}
Which is the desired conclusion.
\end{proof}
\bg
\qquad The following lemma, will be of interest in formulating our main theorem.\bg
%%%%%%%%%%%%%%%%%%%%%%%%%%%%%%%%%%%%%%%%%%%
\begin{proposition} \label{pr43} For $f\in L^1(\tore)$, and a nonzero real $\lambda$, we have
\begin{equation}\label{E:pr431}
\mathcal{L}_\lambda(f)=i\pi \sum_{n\in\ent}(\sgn(n)-\sgn(\lambda))e^{-\abs{n\lambda}}C_n(f),
\end{equation}
where $\sgn(x)$ is the sign of $x$ if $x\ne0$ and $\sgn(0)=0$, and $\big(C_n(f)\big)_{n\in\ent}$ are the exponential Fourier coefficients of $f$.
\end{proposition}
\begin{proof} 

For a nonzero real $\lambda$,  consider the continuous, $2\pi$-periodic, function $g_\lambda$ defined by
\begin{equation}
g_\lambda(t)=\cot\left(\frac{t-i\lambda}{2}\right).
\end{equation}
The exponential Fourier coefficients $(C_n(g_\lambda))_{n\in\ent}$ of $g_\lambda$ are given by
\begin{equation}\label{E:cng}
C_n(g_\lambda)=i(\sgn(\lambda)-\sgn(n))e^{-\abs{n\lambda}}
\end{equation}
\bg
Indeed, let us consider two cases :\bg
\begin{itemize}
\item For $\lambda<0$, we have
\begin{align*}
i g_\lambda(t)&=\frac{1+e^{\lambda+it}}{1-e^{\lambda+it}}=\frac{2}{1-e^{\lambda+it}}-1\\
&=1+\sum_{n=1}^\infty 2e^{n\lambda} e^{int}.
\end{align*}
\item And, for $\lambda>0$, we have
\begin{align*}
i g_\lambda(t)&=\frac{e^{-\lambda-it}+1}{e^{-\lambda-it}-1}=1-\frac{2}{1-e^{-\lambda-it}}\\
&= -1-\sum_{n=1}^\infty2e^{-n\lambda}e^{-int}
= -1-\sum_{n=-\infty}^{-1}2e^{n\lambda}e^{int}.
\end{align*}
\end{itemize}
Combining these two points  proves \eqref{E:cng}.\bg
\qquad Now, since $\abs{C_n(f)}\leq \N{f}_1$ for every $n\in\ent$, we see that 
\[
\sum_{n\in\ent}\abs{C_n(g_\lambda)}\abs{C_n(f)} <+\infty.
\] 
Thus, the continuous, $2\pi$-periodic, function $g_\lambda*f$ is equal  to its Fourier series expansion. In particular,
\[
g_\lambda*f(0)=\frac{1}{2\pi}\int_0^{2\pi}f(t)g_\lambda(-t)\,dt=\sum_{n\in\ent}C_n(g_\lambda)C_n(f)
\]
That is, by \eqref{E:cng},
\[
\frac{1}{2\pi}\int_0^{2\pi}f(t)\cot\left(\frac{t+i\lambda}{2}\right)\,dt
=-i\sum_{n\ent}((\sgn(\lambda)-\sgn(n))e^{-\abs{n\lambda}})C_n(f).
\]
Finally, using Proposition \ref{pr42}, we obtain \eqref{E:pr431}.
\end{proof}
\bg

\qquad Now, combining Corollary \ref{co41}, Proposition \ref{pr41} and Proposition \ref{pr43}, we come to our main Theorem.\bg

\begin{theorem} \label{th41} Consider $f\in L^1(\tore)$ and a positive real $\lambda$, then the following  two integrals are
convergent
\[
\int_0^\infty\frac{f_e(t)}{\lambda^2+t^2}\,dt\qquad\text{and}\qquad \int_0^\infty\frac{tf_o(t)}{\lambda^2+t^2}\,dt
\]
where $f_e$ and $f_o$ are, respectively, the even and odd components of $f$ defined by \eqref{E:4feo}.  Moreover,
\begin{align}
\int_0^\infty\frac{f_e(t)}{\lambda^2+t^2}\,dt&=\frac{\pi}{2 \lambda}\sum_{n\in\ent}e^{-\abs{n}\lambda}C_n(f),\label{E:th411}\\
\int_0^\infty\frac{t f_o(t)}{\lambda^2+t^2}\,dt&=\frac{i\pi}{2}\sum_{n\in\ent}\sgn(n)e^{-\abs{n}\lambda}C_n(f).\label{E:th412}
\end{align}
\end{theorem}
%%%%%%%%%%%%%%%%%%%%%%%%%%%%%%%%%%%%%%%%%%%%%%%
\bg
\begin{corollary}\label{cor2} Consider $f\in L^1(\tore)$ and a positive real $\lambda$.
\begin{enumeratei}
\item If $f$ is \textrm{even}, and its Fourier series expansion is 
$\ds S[f](x)=\frac{a_0}{2}+\sum_{n\geq1}a_n\cos nx$, then
\[
\int_0^\infty\frac{f(t)}{\lambda^2+t^2}\,dt
=\frac{\pi}{2 \lambda}\left(\frac{a_0}{2}+\sum_{n=1}^\infty a_n e^{-n\lambda}\right).
\]
\item If $f$ is  \textrm{odd}, and its Fourier series expansion is 
$\ds S[f](x)= \sum_{n\geq1}b_n\sin nx$, then
\[
\int_0^\infty\frac{t f(t)}{\lambda^2+t^2}\,dt
=\frac{\pi}{2}\sum_{n\geq1}  b_n e^{-n\lambda}.
\]
\end{enumeratei}
\end{corollary}
\bg

\qquad There is a case where Theorem \ref{th41} takes a more practical form. So, let us change our point of view. 
Let $G:\Omega\vers \comp$ be an analytic function
 on a domain $\Omega$, that \textit{contains} the closed
unit disk $\adh{D(0,1)}$. For a positive real number
$\lambda$, we are interested in determining the value of the integrals :
\begin{align}
I_\lambda(G) &=\int_0^\infty\frac{g_c(\theta)}{\lambda^2+\theta^2}\,d\theta=\frac{1}{ \lambda }\int_0^{\pi/2}g_c(\lambda \tan x)\,dx,  \label{E:I}\\
\noalign{\text{and}}\notag\\
J_\lambda(G) &=\int_0^\infty\frac{\theta g_s(\theta)}{\lambda^2+\theta^2}\,d\theta
=\int_0^{\pi/2}g_s( \lambda  \tan x)\tan x\,dx, 
\label{E:J}
\end{align}
where the real variable functions $g_c$ and $g_s$ are, respectively,  the ``cos'' and the ``sin'' components 
of the function $\theta\mapsto G(e^{i\theta})$, defined on $\reel$ by the formul\ae:
\begin{equation}\label{E:feo}
g_c(\theta) =\frac{G(e^{i\theta})+G(e^{-i\theta})}{2}\quad\text{and}\quad
g_s(\theta) =\frac{G(e^{i\theta})-G(e^{-i\theta})}{2i}.
\end{equation}

\qquad The next theorem gives the answer to this question :

\begin{theorem}\label{th1}
 Let $G$ be an anlytic function on a domain $\Omega$ that contains the closed unit disk, then, for every positive real $\lambda$, one has
\begin{align}
\int_0^{\infty}\frac{g_c(x)}{\lambda^2+x^2}\,dx&=\frac{\pi}{2 \lambda }G(e^{- \lambda }),\label{E:I2}\\
\noalign{\text{and}}\notag\\
\int_0^{\infty}\frac{x g_s(x)}{\lambda^2+x^2}\,dx&=\frac{\pi}{2}\left(G(e^{- \lambda })-G(0)\right),
\label{E:J2}
\end{align}
where $g_c$ and $g_s$ are defined by \eqref{E:feo}.
\end{theorem}
\begin{proof}
Indeed, by assumption, there is a power series $\sum_{n=0}^\infty a_nz^n$ with radius of convergence $\rho>1$ such that
\begin{equation}\label{E:G1}
\forall\,z\in\comp,\qquad \abs{z}<\rho\impl G(z)=\sum_{n=0}^\infty a_nz^n
\end{equation}
So, we can define a continuous, $2\pi$-periodic function $f$ by $f(\theta)=G(e^{i\theta})$. 
Comparing \eqref{E:feo} and \eqref{E:4feo} we see that
$f_e(\theta)=g_c(\theta)$ and $f_o(\theta)=i g_s(\theta)$. On the other hand, from \eqref{E:G1} we conclude that
the Fourier coefficients of $f$ are given by $C_n(f)=a_n$ if $n\geq0$ and $C_n(f)=0$ if $n<0$. Thus,   Theorem \ref{th1} proves
\begin{align*}
\int_0^\infty\frac{g_c(t)}{\lambda^2+t^2}\,dt&=\frac{\pi}{2 \lambda}\sum_{n\geq0}a_n e^{-n\lambda}=\frac{\pi}{2 \lambda}G(e^{-\lambda}),\\
\int_0^\infty\frac{t g_s(t)}{\lambda^2+t^2}\,dt&=\frac{\pi}{2}\sum_{n>0} a_ne^{-n\lambda}=\frac{\pi}{2}\left(G(e^{-\lambda})-G(0)\right).
\end{align*}
Which is the desired conclusion. 
\end{proof}

%%%%%%%%%%%%%%%%%%%%%%%%%%%%%%%%%%%
\bg

\section{\sc Examples and Applications}\label{sec4}
\nobreak
\begin{example} \end{example}
\qquad  Let $x$ be a positive number, and consider the analytic function $G$ defined 
on the domaine $\Omega=\comp\setminus\{e^{2x}\}$, by
\begin{equation}\label{E:ex11}
G(z)=\frac{e^{2x}+z}{e^{2x}- z}.
\end{equation}
Here, it is straightforward to check that
\begin{equation}
G(e^{i\theta})=\frac{ \sinh2x +  i \sin\theta  }{ \cosh2x-  \cos\theta},
\end{equation}
so
\begin{align}
g_c(\theta)&= \frac{ \sinh2x }{\cosh2x- \cos\theta},\label{E:ex22}\\
g_s(\theta)&= \frac{ \sin \theta }{\cosh2x-  \cos\theta}.\label{E:ex23}
\end{align}

\qquad Using Theorem \ref{th1} we see that, for $\lambda>0$, we have
\begin{equation}
\int_0^\infty\frac{g_c(\theta)}{\lambda^2+\theta^2}d\theta
=\frac{\pi}{2 \lambda}\cdot\frac{e^{2x}+e^{-\lambda}}{e^{2x}- e^{-\lambda}}
\end{equation}
Thus, from \eqref{E:ex22}, and after making the change of variables $\theta\leftarrow2\theta$ and $\lambda\leftarrow2\lambda$, we get :
\begin{equation*}
\int_0^\infty \frac{ 1}{\cosh2x- \cos2\theta}\cdot\frac{d\theta}{ \lambda^2+ \theta^2}
=\frac{\pi}{2 \lambda\sinh2x}\cdot\frac{e^{2x}+e^{-2\lambda}}{e^{2x}- e^{-2\lambda}}
\end{equation*}
or, equivalently
\begin{equation}
\int_0^\infty \frac{ 1}{ \sinh^2x + \sin^2\theta}\cdot\frac{d\theta}{ \lambda^2+ \theta^2}
=\frac{\pi}{\lambda\sinh2x}\cdot\frac{e^{2x}+e^{-2\lambda}}{e^{2x}- e^{-2\lambda}}
\end{equation}
This can be expressed as follows : For positive $\lambda$ and $\mu$, we have 
\begin{equation}
\int_0^\infty\frac{d\theta}{(\mu^2+\sin^2\theta)(\lambda^2+\theta^2)}
=\frac{\pi}{2\lambda\mu\sqrt{1+\mu^2}}\cdot
\frac{e^{2\lambda}+\left(\sqrt{1+\mu^2}-\mu\right)^2 }{e^{2\lambda}-\left(\sqrt{1+\mu^2}-\mu\right)^2}.
\end{equation}
In particular, choosing $\lambda=\mu=1$ we find the integral \eqref{E:eq1} that we used to introduce our discussion.
\bg
\qquad On the other hand, using $g_s$ from \eqref{E:ex23}, and the second part of Theorem \ref{th1}, we find that, for $\lambda>0$, we have

\begin{equation}
\int_0^\infty\frac{ \theta \sin \theta }{(\cosh2x-\cos\theta)(\lambda^2+\theta^2)}\,d\theta
=\frac{\pi}{e^{2x+\lambda}- 1},
\end{equation}
or, equivalently,  for $\mu>1$ and $\lambda>0$ :
\begin{equation}
\int_0^\infty\frac{ \theta \sin \theta }{(\mu-\cos\theta)(\lambda^2+\theta^2)}\,d\theta
=\frac{\pi}{e^{\lambda} (\mu+\sqrt{\mu^2-1})-1}.
\end{equation}
The integrals in this example are to be compared with \cite[formul\ae~ 3.792(10) and 3.792(13), page 450]{grad}.
\bg
%%%%%%%%%%%%%%%%%%%%%%%%%%%%%%%%%%%%%%%%%%%
\begin{example}\end{example}

\qquad In our second example we consider the even function $f\in L^1(\tore)$ defined by
\begin{equation}
f(t)=-\ln \abs{\sin(t/2)}=-\frac{1}{2}\ln\sin^2(t/2),\quad\text{for $t\notin2\pi\ent$.}
\end{equation}
It is well-known, (see for instance \cite[Chapter~3,~\$\,14]{tol},) that $f$ has the following Fourier series expansion 
\[
S[f](t)=\ln2+\sum_{n=1}^\infty\frac{\cos nt}{n}
\]
Thus, applying Corollary \ref{cor2} we obtain that for every $\lambda>0$ we have
\[
\int_0^\infty\frac{\ln(\sin^2(t/2))}{\lambda^2+t^2}\,dt
=-\frac{\pi}{\lambda}\left(\ln2+\sum_{n=1}^\infty\frac{e^{-n\lambda}}{n}\right)
=-\frac{\pi}{\lambda}\left(\ln 2-\ln(1-e^{-\lambda})\right).
\]
The change of variables $t\leftarrow2t$ and $\lambda\leftarrow2\lambda$ yields the  following result:
\begin{equation}
\text{ for $\lambda>0$,}\qquad\int_0^\infty\frac{\ln(\sin^2t)}{ \lambda^2+ t^2}\,dt=\frac{\pi}{\lambda}\ln\left(\frac{1-e^{-2\lambda}}{2}\right).
\end{equation}

\qquad Our next example is a generalization of an old problem.
\bg
%%%%%%%%%%%%%%%%%%%%%%%%%%%%%%%%%%%%%%%%%%%%
\begin{example} A Generalization of A Problem of Narayana Aiyar.\end{example}

\qquad Let $t_1,t_2,\ldots,t_n$ and $a$ be positive real numbers, such that
\[
0<t_1\leq t_2\leq \ldots\leq t_n<a.
\]
Consider the meromorphic function $G$ defined by
\begin{equation}
G(z)=\frac{1}{(a-t_1z)(a-t_2 z)\cdots(a-t_n z)}
\end{equation}
Clearly, $G$ is analytic in the domain $\Omega=\comp\setminus\left\{\frac{a}{t_k}:1\leq k\leq n\right\}$
that contains the closed unit disk.\bg

 For a given real $\theta$, let
\begin{equation}\label{E:rf}
\phi_k(\theta)=\arctan\left(\frac{t_k\sin\theta}{a-t_k\cos\theta}\right),\qquad
 \rho_k(\theta)=\sqrt{a^2-2t_k a\cos\theta+t_k^2}.
\end{equation}
To simplify the notation, we will simply write $\phi_k$ and $\rho_k$ to denote $\phi_k(\theta)$ and $\rho_k(\theta)$
respectively. It is clear that 
\begin{equation}
a-t_k e^{i\theta}=\rho_k e^{-i\phi_k}, \qquad\text{for $1\leq k\leq n$}.
\end{equation}
Thus,
\begin{equation}
G(e^{i\theta})=\frac{e^{i(\phi_1+\cdots+\phi_n)}}{\rho_1\cdots \rho_n}.
\end{equation}
On the other hand, since  $G(e^{-i\theta})=\adh{G(e^{i\theta})}$, we see immediately that
\begin{equation}
g_c(\theta)=\frac{\cos(\phi_1+\cdots+\phi_n)}{\rho_1\cdots \rho_n},\qquad\hbox{and}\qquad
g_s(\theta)=\frac{\sin(\phi_1+\cdots+\phi_n)}{\rho_1\cdots \rho_n}.
\end{equation}
Therefore, using Theorem \ref{th1}, with $\lambda=1$, we obtain 
$I_1(G)-J_1(G)=\frac{\pi}{2}G(0)$, that is
\begin{equation}\label{E:aiyar}
\int_0^\infty\frac{\cos(\phi_1+\cdots+\phi_n)-\theta \sin(\phi_1+\cdots+\phi_n)}{\rho_1\cdots \rho_n}\frac{d\theta}{1+\theta^2}=\frac{\pi}{2a^n}.
\end{equation}

\qquad The evaluation of the integral \eqref{E:aiyar}, when $a=1$ and $t_k=k r$ for some $0<r<1/n$, is an unsolved problem
proposed by Narayana Aiyar in the begining of the twentieth century \cite{berndt}, while the generalization,
corresponding to $a>0$ and $t_k=k r$ for some $0<r<a/n$, is a problem proposed by M.~D.~ Hirchhorn \cite{Hirch}.
\bg
\qquad Note that Theorem \ref{th1} yields the following more precise results, valid for  $\lambda>0$ :
\begin{align}
\int_0^\infty\frac{\cos(\phi_1+\cdots+\phi_n)}{\rho_1\cdots \rho_n}\cdot\frac{d\theta}{\lambda^2+\theta^2}&=
\frac{\pi}{2\lambda}\prod_{k=1}^n\frac{1}{a-t_ke^{-\lambda}},\\
\noalign{\text{and}}\notag\\
\int_0^\infty\frac{\sin(\phi_1+\cdots+\phi_n)}{\rho_1\cdots \rho_n}\cdot\frac{\theta\, d\theta}{\lambda^2+\theta^2}&=
\frac{\pi}{2}\left(\prod_{k=1}^n\frac{1}{a-t_ke^{-\lambda}}-\frac{1}{a^n}\right).
\end{align}
where $t_1,\ldots,t_n$ are real numbers from the interval $(0,a)$ and the $\phi_k$'s and $\rho_k$'s are defined by \eqref{E:rf}.
\bg

%%%%%%%%%%%%%%%%%%%%%%%%%%%%%%%%%%%%%%%%%%%%
\begin{example} Bernoulli Polynomials and Polylogarithms.\end{example}
\qquad In this example we consider the sequence $(B_m)_{m\geq0}$ of \textit{Bernoulli polynomials}. They can be 
defined via the generating function
\begin{equation}
\frac{te^{tx}}{e^t-1}=\sum_{m=0}^\infty B_m(x)\frac{t^m}{m!},
\end{equation}
or recursively by $B_0 =1$ and 
\begin{equation}\label{E:bn}
\forall\,n>0,\qquad B_n^\prime=B_{n-1},\quad\text{and }\quad \int_0^1B_n(t)\,dt=0.
\end{equation}
In particular, 
\begin{equation}
B_1(X)=X-\frac12,\quad B_2(X)=X^2-X+\frac16,\qquad B_3(X)=X(X-\frac12)(X-1).
\end{equation}
\bg

\qquad Now, denote by $\widetilde{B}_m$ the $2\pi$-periodic function defined by
\begin{equation}
\widetilde{B}_m(x)=B_m\left(\left\{\frac{x}{2\pi}\right\}\right),\quad\text{where $\{u\}$ is the fractional part of $u$.}
\end{equation}
\bg
\qquad For $m \geq1$, the Fourier series expansion of $ \widetilde{B}_m$, is well-known and  easy to find
(using the recursive definition \eqref{E:bn}, see, for example \cite[Chapter~23]{abr}.) We have
\begin{align}
S[\widetilde{B}_{2m-1}](x)&=\frac{(-1)^m2(2m-1)!}{(2\pi)^{2m-1}}\sum_{n=1}^\infty\frac{\sin(nx)}{n^{2m-1}}\\
S[\widetilde{B}_{2m}](x)&=\frac{(-1)^m2(2m)!}{(2\pi)^{2m}}\sum_{n=1}^\infty\frac{\cos(nx)}{n^{2m}}
\end{align} \bg
\qquad Applying Corollary \ref{cor2} we conclude that, for $m\geq1$ and $\lambda>0$, we have
\begin{align*}
\int_0^\infty\frac{x}{\lambda^2+x^2}\widetilde{B}_{2m-1}(x)\,dx&=
\frac{(-1)^m(2m-1)!}{2(2\pi)^{2m-2}}\sum_{n=1}^\infty\frac{e^{-n\lambda}}{n^{2m-1}}\\
\int_0^\infty\frac{1}{\lambda^2+x^2}\widetilde{B}_{2m}(x)\,dx&=
\frac{(-1)^m(2m)!}{2\lambda(2\pi)^{2m-1}}\sum_{n=1}^\infty\frac{e^{-n\lambda}}{n^{2m}}
\end{align*}
The change of variables $x\leftarrow2\pi x$ and $\lambda\leftarrow2\pi\lambda$ yields the  following result, for $m\geq1$ and $\lambda>0$ :
\begin{align}
\int_0^\infty\frac{x B_{2m-1}(\{x\})}{\lambda^2+x^2} \,dx&=(-1)^m\,\frac{(2m-1)!}{2(2\pi)^{2m-2}}\,\li_{2m-1}(e^{-2\pi\lambda})\\
\int_0^\infty\frac{B_{2m}(\{x\})}{\lambda^2+x^2} \,dx&=(-1)^m\,\frac{(2m)!}{2(2\pi)^{2m-1} }\,\frac{\li_{2m}(e^{-2\pi\lambda})}{\lambda}
\end{align}
where the function $\li_k$ is the polylogarithm of order $k$. The function $\li_k$ is 
defined on open unit disk by the series $\sum_{n=1}^\infty z^n/n^k$. (For an extensive account of the polylogarithms see \cite{lew}.)

In particular, since $B_1(X)=X-\frac{1}{2}$, we see that $\lambda>0$, we have 
\begin{equation}\label{E:Li1}
\int_0^\infty\frac{x (\{x\}-1/2)}{\lambda^2+x^2} \,dx =\frac{1}{2}\ln(1-e^{-2\pi\lambda}).
\end{equation}
Also, from the expressions of $B_2$ and $B_3$ we conclude that, for $\lambda>0$, we have
\begin{align}
\int_0^\infty\frac{\{x\}(1-\{x\})}{\lambda^2+x^2} \,dx& =\frac{\pi}{12\lambda}+
\frac{1}{2\pi\lambda }\li_{2}(e^{-2\pi\lambda}).\\
\int_0^\infty\frac{x\{x\}(\{x\}-1)(\{x\}-1/2)}{\lambda^2+x^2} \,dx& =
\frac{3}{4\pi^2 }\li_{3}(e^{-2\pi\lambda}).\label{E:Li3}
\end{align}
\bg
Adding one forth of \eqref{E:Li1} to \eqref{E:Li3} we see that, for $\lambda>0$, we have
\begin{equation}
\int_0^\infty\frac{x(\{x\}-1/2)^3}{\lambda^2+x^2} \,dx  =\frac{1}{8}\ln(1-e^{-2\pi\lambda})+ \frac{3}{4\pi^2 }\li_{3}(e^{-2\pi\lambda}).
\end{equation}
\bg
%%%%%%%%%%%%%%%%%%%%%%%%%%%%%%%%%%%%%%%%%%%%%%%%%%%%%%

\begin{fexample}\end{fexample}
\qquad We will end our discussion by citing some results that can
be proved using the methods of this paper, without presenting the details :
\begin{itemize}
\item For positive real numbers $\mu$ and $\lambda$, we have
\begin{equation}
\int_0^\infty\frac{x\arctan(\mu \sin x)}{x^2+\lambda^2}\,dx
=\frac{\pi}{2}\ln\frac{\mu e^\lambda+\sqrt{1+\mu^2}-1}{\mu e^\lambda-\sqrt{1+\mu^2}+1}.
\end{equation}

\item For positive real numbers $\mu$ and $\lambda$, we have
\begin{equation}
\int_0^\infty\frac{\arctan(\mu \cos x)}{x^2+\lambda^2}\,dx
=\frac{\pi}{\lambda}\arctan\frac{\sqrt{1+\mu^2}-1}{\mu e^\lambda}.
\end{equation}

\item For real numbers $\mu$ and $\lambda$ such that $\mu\in(-1,1)$ and $\lambda>0$, we have
\begin{equation}
\int_0^\infty\frac{1}{x^2+\lambda^2}\ln \frac{1+\mu \cos x}{1-\mu \cos x}\,dx
=\frac{\pi}{\lambda}\ln\frac{\mu e^\lambda-\sqrt{1-\mu^2}+1}{\mu e^\lambda+\sqrt{1-\mu^2}-1}.
\end{equation}

\item For real numbers $\mu$ and $\lambda$ such that $\mu\in(-1,1)$ and $\lambda>0$, we have
\begin{equation}
\int_0^\infty\frac{x}{x^2+\lambda^2}\ln \frac{1+\mu \sin x}{1-\mu \sin x}\,dx
=2\pi\arctan\frac{1-\sqrt{1-\mu^2}}{\mu e^\lambda}.
\end{equation}

\item For  $z\in \comp$ and a positive $\lambda$, we have
\begin{equation}
\int_0^\infty\frac{e^{iz\cos x}}{x^2+\lambda^2} \,dx
=\frac{\pi}{\lambda}\left(J_0(z)+2\sum_{n=1}^\infty i^n J_n(z)e^{-n\lambda}\right),
\end{equation}
where $J_n$ is the well-known Bessel function of the first kind of order $n$.

\item For positive $\lambda$, we have
\begin{equation}
\int_0^\infty\frac{\ln(\tan^2x)}{x^2+\lambda^2} \,dx
=\frac{\pi}{\lambda}\ln(\tanh \lambda).
\end{equation}

\end{itemize}

%%%%%%%%%%%%%%%%%%%%%%%%%%%%%%%%%%%

\end{document}